\documentclass[a4paper,12pt]{article}

\usepackage{amsmath, amsthm, amsfonts, amssymb, dsfont,   accents}
 
\newtheorem{lemma}{Lemma}[section]
\newtheorem{theorem}{Theorem}[section]

\theoremstyle{remark}
\newtheorem{remark}{Remark}[section]

\numberwithin{equation}{section}

\def\Om{\Omega}
\def\om{\omega}
\def\e{\varepsilon}
\def\g{\gamma}
\def\G{\Gamma}
\def\l{\lambda}
\def\p{\partial}
\def\D{\Delta}

\def\E{\mbox{\rm e}}
\def\a{\alpha}
\def\b{\beta}

\def\d{\delta}
\def\L{\Lambda}
\def\z{\zeta}

\def\di{\,d}

\def\Op{\mathcal{H}}
\def\S{\mathcal{S}}

\newcommand{\EE}{\mathbb{E}}
\newcommand{\NN}{\mathds{N}}
\newcommand{\PP}{\mathbb{P}}

\newcommand{\RR}{\mathds{R}}
\newcommand{\ZZ}{\mathds{Z}}

\def\bc{\mathcal{B}}

\def\Dom{\mathfrak{D}}

\def\Wl{W^{\mathrm{l}}}
\def\Ws{W^{\mathrm{s}}}
\def\Wd{W^{\mathrm{dlt}}}
\def\Wle{W^{\mathrm{loc}}_*}
\def\Wse{W^{\mathrm{osc}}_*}
\def\Wloc{W^{\mathrm{loc}}}
\def\Wosc{W^{\mathrm{osc}}}
\def\Tosc{T^{\mathrm{osc}}}

\def\Opl{\mathcal{H}^{\e,\mathrm{loc}}(\om)}
\def\Ops{\mathcal{H}^{\e,\mathrm{osc}}(\om)}
\def\Opd{\mathcal{H}^{\e,\mathrm{dlt}}(\om)}
\def\fh{\mathfrak{h}_{\mathrm{dlt}}}
\def\OplN{\mathcal{H}_{\a,N}^{\e,\mathrm{loc}}(\om)}
\def\OpsN{\mathcal{H}_{\a,N}^{\e,\mathrm{osc}}(\om)}
\def\OpdN{\mathcal{H}_{\a,N}^{\e,\mathrm{dlt}}(\om)}
\def\lN#1{\l_{\a,N}^{\e,\mathrm{#1}}(\om)}
\def\pL{\mathcal{L}}
\def\Opg{\mathcal{H}^\e_{\a,N}(\om)}
\def\lg{\l^\e_{\a,N}(\om)}
\def\Vg#1{\mathcal{V}^{\e,\mathrm{#1}}_{\a,N}(\om)}

 \DeclareMathOperator{\spec}{\sigma}

\DeclareMathOperator{\dist}{dist} \DeclareMathOperator{\supp}{supp}

\begin{document}
\allowdisplaybreaks

\begin{center}
\textbf{\large Initial length scale estimate for waveguides with some random singular potentials}

\medskip

 D.I. Borisov$^{1,2,3}$, R.Kh. Karimov$^{2}$, T.F. Sharapov$^{2}$

\end{center}

\begin{quote}
{\it 1) Institute of Mathematics CC USC RAS,
\\
Chernyshevskii str., 112,
450008, Ufa, Russia
\\
2) Bashkir State Pedagogical University named after M. Akhmulla,
\\
October rev. st., 3a, 450000, Ufa, Russia
\\
3)  University of Hradec Kr\'alov\'e,
  Rokitanskeho, 62,
  \\
   50003, Hradec Kr\'alov\'e, Czech Republic
\\
E-mails: borisovdi@yandex.ru, karimov\_ramis\_92@mail.ru, stf0804@mail.ru
}
\end{quote}

\begin{abstract}
In this work we consider three examples of random singular perturbations in multi-dimensional models of waveguides. These perturbations are described by a large potential supported on a set of a small measure, by a compactly supported fast oscillating potential, and by a delta-potential. In all cases we prove initial length scale estimate.
\end{abstract}
{\small
\begin{quote}
\noindent{\bf Keywords:} random operator, initial length scale estimate, perturbation, small parameter, spectral localization.

\noindent{\bf Mathematics Subject Classification: }{35P15, 35C20, 35B25, 60H25, 82B44}
\end{quote}}

\section{Introduction}

One of the approaches for describing wave processes in disordered media are random Hamiltonians, which are elliptic operators in unbounded domains depending on a countably many independent identically distributed random variables. Such operators are quite intenstively studied. One of issues in question is the spectral localization. The latter means that the whole spectrum or a part of it are pure point with the probability one. There are many works where such property of the spectrum was studied for numerous particular examples, see, for instance, \cite{MH84}--\cite{Veselic-02a}, and the references therein. One of the known ways for proving the spectral localization is the multiscale analysis, \cite{MH84}, \cite{FS83}. It is based on a certain induction whose basis is the initial length scale estimate.

In paper \cite{BGV15}, there was proposed a general approach for proving initial length scale estimate for operators with small random perturbations. The perturbations were described by abstract symmetric operators being small w.r.t. the original unperturbed one. Under minimal conditions for the perturbations, the initial length scale estimate was proven at the bottom of the spectrum. Such general approach allowed the authors to consider various examples both known and new.

The present paper is a continuation of work \cite{BGV15}. We consider three examples of random perturbations. Each of them is not regular, i.e., small w.r.t. the original unperturbed operator. Moreover, these perturbations are singular in some sense. At that same time, we show that the results of \cite{BGV15} on initial length scale estimate can be extended for the considered perturbations.
This is the main result of the present paper.

\section{Formulation of problem and main results}

Let $x=(x',x_{n+1})$, $x'=(x_1,\ldots,x_n)$ be Cartesian coordinates in $\RR^{n+1}$ and $\RR^n$, respectively, $n\geqslant 1$. By $\Pi$ we denote an infinite multi-dimensional layer of width $d>0$:
\begin{equation*}
\Pi:=\{x: \, 0<x_{n+1}<d\}.
\end{equation*}
In layer $\Pi$, we consider the operator
\begin{equation}\label{2.1}
\Op^0:=-\D+V_0,
\end{equation}
where $V_0=V_0(x_{n+1})$ is a bounded measurable potential. As the boundary condition on $\p\Pi$, we choose the Dirichlet or Neumann condition:
\begin{equation}\label{2.2}
\bc u=0
\end{equation}
on $\p\Pi$, where $\bc u=u$ or $\bc u=\frac{\p u}{\p x_{n+1}}$. We do not exclude the situation, when on the upper and lower boundaries of $\p\Pi$ the boundary conditions of different types are imposed.

Operator $\Op^0$ is considered as unbounded in space $L_2(\Pi)$ on the domain $\Dom(\Op^0):=\{u\in H^2(\Pi): \ \text{condition (\ref{2.2}) is satisfied on $\p\Pi$}\}$.

Let us describe random perturbation. Let $\G$ be a periodic lattice in $\RR^n$ with the periodicity cell $\square'$ and $\square:=\{x:\ x'\in\square',\, 0<x_{n+1}<d\}$. By $\Wl=\Wl(x')$ we denote a continuous compactly supported function $\RR^n$, and $\Ws=\Ws(x,\xi)$, $\xi=(\xi_1,\ldots,\xi_{n+1})$, stands for a function in $\RR^{2n+2}$ 1-periodic w.r.t. each of the variables $\xi_i$, $i=1,\ldots,n$, having a zero mean
\begin{equation}\label{2.4}
\int\limits_{(0,1)^{n+1}} \Ws(x,\xi)\di\xi=0\quad \text{for each} \quad x\in\RR^{n+1},
\end{equation}
and compactly supported w.r.t. $x$:
\begin{equation}\label{2.3}
\supp \Ws(\cdot,\xi)\subseteq M\subset \square \quad\text{для всех}\quad \xi\in\RR^{n+1},
\end{equation}
where $M$ is a some fixed set. We assume the following smoothness for function $\Ws$:
\begin{equation}\label{2.21}
\frac{\p^{|\a|+|\b|}\Ws}{\p x^\a \p\xi^\b} \in C(\RR^{2n+2}),\quad \a,\b\in\ZZ_+^n,\quad |\a|\leqslant 3, \quad |\b|\leqslant 1.
\end{equation}

By $W=W(x)$ we denote a continuous function compactly supported in $\square$:
\begin{equation*}
\supp W\Subset \square.
\end{equation*}
Let $S\Subset\square$ be a closed $C^4$-manifold of codimension $1$, $\nu$ be the normal to $S$ outward w.r.t. the domain enveloped by manifold $S$, $\Wd\in C^3(S)$ be a real non-negative function on $S$.

By $\e$ we denote a small positivi parameter. We let:
\begin{equation}\label{2.5}
\begin{aligned}
&\Wloc(x',\e):=\e^{-a} \Wl\left(\frac{x'}{\e}\right), && \e>0,
\\
& \Wosc(x,\e):=\e^{-a} \Ws\left(x,\frac{x}{\e}\right)+\e^{2-2a} W(x), && \e>0,
\\
&\Wloc(x',0):=0,\qquad \Wosc(x',0):=0,
\end{aligned}
\end{equation}
where $0\leqslant a<1$ is a given number.

Let $\om=(\om_k)_{k\in \G}$ be a sequence of independent identically distributed random variables with the values in segment $[0,1]$; the associated distribution measure is denoted by $\mu$. We assume that this measure is defined on $[0,1]$. By $\PP:=\bigotimes_{k\in\G} \mu$ we denote the product of the measures on space $\Om:=\times_{k\in\G} [0,1]$. The elements of the latter space are sequences $(\om_k)_{k\in\G}$. By $\EE(\cdot)$ we denote the expectation value of a random variable w.r.t. probability $\PP$.

The first two types of random perturbation are described by the operators:
\begin{align}
&\Opl:=\Op^0+\sum\limits_{k\in\G} \Wloc(\cdot-k,\e\om_k),
\label{2.6}
\\
&\Ops:=\Op^0+\sum\limits_{k\in\G} \Wosc(\cdot-k,x_{n+1},\e \om_k).
\label{2.7}
\end{align}
The third type corresponds to an operator with a small delta-interaction:
\begin{equation}\label{2.8}
\Opd:=\Op^0+\sum\limits_{k\in\G} \e\om_k \Wd(\cdot-k) \d(\cdot-S_k),
\end{equation}
where $S_k$ is a shift of manifold $S$ by $k$, namely, $S_k:=\{x: (x'-k,x_{n+1})\in S\}$. In all three cases the boundary condition on $\p\Pi$ is described by identity (\ref{2.2}). Notion (\ref{2.8}) is formal for indicating the operator in $L_2(\Pi)$ associated with the sesquilinear form
\begin{equation}\label{2.22}
\fh(u,v):=(\nabla u,\nabla v)_{L_2(\Pi)}+\sum\limits_{k\in\G} \e \om_k \big(\Wd(\cdot-k)u,v\big)_{L_2(S_k)} \quad \text{in}\quad L_2(\Pi).
\end{equation}
The domain of this form is the set of functions in $H^1(\Pi)$ having zero trace on the Dirichlet part of boundary $\p\Pi$. One more equivalent description of operator $\Opd$ is operator $-\D+V_0$ in $\Pi$ with boundary condition (\ref{2.2}) on $\p\Pi$ and the boundary condition
\begin{equation}\label{2.23}
[u]_{S_k}=0,\quad \left[\frac{\p u}{\p\nu}\right]_{S_k}=bu\big|_{S_k}, \quad k\in\G,
\end{equation}
where $[v]_{S_k}=v\big|_{S_k+0}-v\big|_{S_k-0}$ is the jump of function $v$ at $S_k$ being the difference of the values on the external and internal sides of $S_k$.

The main aim of the present work is to obtain initial length scale estimate for operators $\Opl$, $\Ops$, $\Opd$.

To formulate the main results, we shall make use of additional auxiliary notations. Given $\a\in\G$, $N\in\NN$, the symbol $\Pi_{\a,N}$ stands for a piece of layer $\Pi$:
\begin{equation*}
\Pi_{\a,N}:=\bigg\{x:\ x'=\a+\sum\limits_{i=1}^{n}a_i e_i,\ a_i\in(0,N),\ 0<x_{n+1}<d\bigg\}.
\end{equation*}
Here $e_i$, $i=1,\ldots,n$ is the basis of lattice $\G$, i.e.,
\begin{equation*}
\G:=\bigg\{x:\ x'=\sum\limits_{i=1}^{n}a_i e_i,\ a_i\in\ZZ\bigg\}.
\end{equation*}
We also denote
\begin{equation*}
\G_{\a,N}:=\bigg\{x'\in\G:\ x'=\a+\sum\limits_{i=1}^{n}a_i e_i,\ a_i=0,1,\ldots,N-1\bigg\}.
\end{equation*}
We observe that
\begin{equation*}
\overline{\Pi_{\a,N}}:=\bigcup\limits_{k\in\G_{\a,N}} \overline{\square_k}.
\end{equation*}

By $\OplN$, $\OpsN$, $\OpdN$ we denote operators which are introduced in the same way as $\Opl$, $\Ops$, $\Opd$, but on set $\Pi_{\a,N}$ with additional Neumann condition on the lateral boundary. Namely, $\OplN$, $\OpsN$ are the operators
\begin{equation*}
-\D+V_0+\sum\limits_{k\in\G_{\a,N}} \Wloc(\cdot-k,\e\om_k) \quad\text{и}\quad -\D+V_0+\sum\limits_{k\in\G_{\a,N}} \Wosc(\cdot-k,x_{n+1},\e\om_k)
\end{equation*}
in $\Pi$ subject to boundary condition (\ref{2.2}) on the upper and lower boundaries and subject to the boundary condition
\begin{equation}\label{2.9}
\frac{\p u}{\p\nu}=0\quad \text{на}\quad \p\Pi_{\a,N}\setminus\p\Pi,
\end{equation}
where $\nu$ is the outward norm to the boundary. Operator $\OpdN$ is introduced by the (formal) identity
\begin{equation*}
\OpdN:=-\D+V_0 + \sum\limits_{k\in\G_{\a,N}} \e \om_k \Wd(\cdot-k) \d(\cdot-S_k)
\end{equation*}
in $\Pi_{\a,N}$ subject to boundary condition (\ref{2.2}) on upper and lower boundaries and subject to boundary condition (\ref{2.9}). One can define it rigorously by means of sesquilinear form similar to (\ref{2.22}) or by means of boundary conditions (\ref{2.23}) for $k\in\G_{\a,N}$.

Let $\lN{\sharp}$, $\sharp=\mathrm{loc}, \mathrm{osc}, \mathrm{dlt}$ be the minimal eigenvalue of operators $\OplN$, $\OpsN$, $\OpdN$, and $\L_0$ be the minimal eigenvalue of the operator
\begin{equation*}
-\frac{d^2}{dx_{n+1}^2}+V_0\quad \text{на}\quad (0,d)
\end{equation*}
subject to boundary condition (\ref{2.2}) at the end-points. The eigenfunction associated $\L_0$ is denoted by $\psi_0=\psi_0(x_{n+1})$ and it is assumed to be normalized in $L_2(0,d)$.

Our first result provides an important lower deterministic estimate for the difference $\lN{\sharp}-\L_0$.

\begin{theorem}\label{th1loc}
Suppose that $n=1$, the origin lies in $\square'$ and
\begin{equation}\label{2.19}
\int\limits_{\RR} \Wl(\z) \di\z>0.
\end{equation}
Then there exist positive constants $c_1$, $c_2$, $N_1$ such that for
\begin{equation}\label{2.10}
N\geqslant N_1\quad \text{and} \quad 0<\e<\frac{c_1}{N^\frac{8}{1-a}}
\end{equation}
the estimate
\begin{equation}\label{2.11}
\lN{loc}-\L_0\geqslant \frac{c_2\e^{1-a}}{N}\sum\limits_{k\in \G_{\a,N}} \om_k^{1-a}
\end{equation}
holds true.
\end{theorem}

By $\Ws_*=\Ws_*(x,\xi)$ we denote the solution to the equation
\begin{equation}\label{2.12}
\D_{\xi}\Ws_*(x,\xi)=\Ws(x,\xi),\quad \xi\in(0,1)^{n+1}
\end{equation}
subject to periodic boundary conditions obeying the orthogonality condition:
\begin{equation}\label{2.20}
\int\limits_{(0,1)^{n+1}} \Ws_*(x,\xi)\di\xi=0,\quad x\in\RR^{n+1}.
\end{equation}
By identity (\ref{2.4}), such problem for $\Ws_*$ is uniquely solvable. Moreover, it follows from (\ref{2.21}) that function $\Ws_*$ has at least the same smoothness as $\Ws$.

\begin{theorem}\label{th1osc}
Suppose that $n\geqslant 1$,
\begin{equation}\label{2.18}
\int\limits_{\square} W(x)\psi_0^2(x_{n+1})\di x-\int\limits_{\square} \di x\, \psi_0^2(x_{n+1}) \int\limits_{(0,1)^{n+1}} |\nabla_{\xi} \Ws_*(x,\xi)|^2\di\xi>0.
\end{equation}
Then there exist positive constants $c_1$, $c_2$, $N_1$ such that for
\begin{equation}\label{2.13}
N\geqslant N_1\quad \text{и} \quad 0<\e<\frac{c_1}{N^\frac{4}{1-a}}
\end{equation}
the estimate
\begin{equation}\label{2.14}
\lN{osc}-\L_0\geqslant \frac{c_2\e^{2-2a}}{N^n}\sum\limits_{k\in \G_{\a,N}} \om_k^{2-2a}
\end{equation}
holds true.
\end{theorem}

\begin{theorem}\label{th1dlt}
Suppose that $n\geqslant 1$,
\begin{equation}\label{2.15}
\int\limits_{S} \Wd(x)\psi_0^2(x_{n+1})\di S>0.
\end{equation}
Then there exist positive constants $c_1$, $c_2$, $N_1$ such that for
\begin{equation}\label{2.16}
N\geqslant N_1\quad \text{и} \quad 0<\e<\frac{c_1}{N^8}
\end{equation}
the estimate
\begin{equation}\label{2.17}
\lN{dlt}-\L_0\geqslant \frac{c_2\e}{N^n}\sum\limits_{k\in \G_{\a,N}} \om_k
\end{equation}
holds true.
\end{theorem}

Our next deterministic results describe Combes-Thomas estimates for the considered operators. We denote by $\chi_B=\chi_B(x)$ the characteristic function of a set $B\subseteq\Pi$, by $\|\cdot\|_{X\to Y}$ we denote the norm of an operator acting from a Banach space $X$ into a Banach space $Y$, $\spec(\cdot)$ stands for the spectrum of an operator.

\begin{theorem}\label{th2loc}
Suppose that $\a, \b_1, \b_2\in\G$, $m_1, m_2\in\NN$ are such that ${B_1:=\Pi_{\b_1,m_1}\subset \Pi_{\a,N}}$, ${B_2:=\Pi_{\b_2,m_2}\subset \Pi_{\a,N}}$ and the assumption of Theorem~\ref{th1loc} is satisfied. Then there exists $N_2\in\NN$ such that for ${N\geqslant N_2}$ the estimate
\begin{equation*}
\|\chi_{B_1}(\OplN -\l)^{-1}\chi_{B_2}\|_{L_2(\Pi_{\a,N})\to L_2(\Pi_{\a,N})} \leqslant \frac{C_1}{\d} \E^{-C_2\d \dist(B_1, B_2)}
\end{equation*}
holds true, where $\d:=\dist(\l,\spec(\OplN))>0$, $C_1$, $C_2$ are positive constants independent of $\e$, $\a$, $N$, $\d$, $\b_1$, $\b_2$, $m_1$, $m_2$, $\l$.
\end{theorem}

\begin{theorem}\label{th2osc}
Suppose that $\a, \b_1, \b_2\in\G$, $m_1, m_2\in\NN$ are such that ${B_1:=\Pi_{\b_1,m_1}\subset \Pi_{\a,N}}$, ${B_2:=\Pi_{\b_2,m_2}\subset \Pi_{\a,N}}$ and the assumption of Theorem~\ref{th1osc} is satisfied. Then there exists $N_2\in\NN$ such that for ${N\geqslant N_2}$ the estimate
\begin{equation*}
\|\chi_{B_1}(\OpsN -\l)^{-1}\chi_{B_2}\|_{L_2(\Pi_{\a,N})\to L_2(\Pi_{\a,N})} \leqslant \frac{C_1}{\d} \E^{-C_2\d \dist(B_1, B_2)},
\end{equation*}
holds true, where $\d:=\dist(\l,\spec(\OpsN))>0$, $C_1$, $C_2$ are positive constants independent of $\e$, $\a$, $N$, $\d$, $\b_1$, $\b_2$, $m_1$, $m_2$, $\l$.
\end{theorem}

\begin{theorem}\label{th2dlt}
Suppose that $\a, \b_1, \b_2\in\G$, $m_1, m_2\in\NN$ are such that ${B_1:=\Pi_{\b_1,m_1}\subset \Pi_{\a,N}}$, ${B_2:=\Pi_{\b_2,m_2}\subset \Pi_{\a,N}}$, and the assumption of Theorem~~\ref{th1dlt} is satisfied. Then there exists $N_2\in\NN$ such that for ${N\geqslant N_2}$ the estimate
\begin{equation*}
\|\chi_{B_1}(\OpdN -\l)^{-1}\chi_{B_2}\|_{L_2(\Pi_{\a,N})\to L_2(\Pi_{\a,N})} \leqslant \frac{C_1}{\d} \E^{-C_2\d \dist(B_1, B_2)},
\end{equation*}
holds true, where $\d:=\dist(\l,\spec(\OpdN))>0$, $C_1$, $C_2$ are positive constants independent of $\e$, $\a$, $N$, $\d$, $\b_1$, $\b_2$, $m_1$, $m_2$, $\l$.
\end{theorem}

Our first probabilistic result is presented in the next three theorems.

\begin{theorem}\label{th3loc}
Suppose that $\g\in\NN$, $\g\geqslant 17$ and the assumption of Theorem~\ref{th1loc} is satisfied. Then the interval
\begin{equation*}
I_N:=\left[\frac{c_{3}}{\big(\EE(\omega_k^{\frac{1-a}{2}})\big)^{\frac{2}{1-a}} N^{\frac{8}{1-a}}}, \frac{c_{1}}{N^{\frac{8}{\gamma (1-a)}}}\right],
\quad c_{3}:= \frac{2^{\frac{2}{1-a}}}{c_{2}^{\frac{1}{1-a}}},
\end{equation*}
is non-empty $N\geqslant N_1$, where $N_1$, $c_1$, $c_2$ are from Theorem~\ref{th1loc}. For $N\geqslant N_1$ and $\e\in I_N$ the estimate
\begin{equation*}
\PP\left(\omega\in\Om:\, \lN{loc}-\L_0\leqslant N^{-\frac{1}{2}}\right)\leqslant N^{\left(1-\frac{1}{\g}\right)}\E^{-c_{4}N^{\frac{1}{\g}}}
\end{equation*}
holds true, where constant $c_4>0$ depends only on distribution measure $\mu$.
\end{theorem}

\begin{theorem}\label{th3osc}
Suppose that $\g\in\NN$, $\g\geqslant 17$ and the assumption of Theorem~\ref{th1osc} is satisfied. Then the interval
\begin{equation*}
I_N:=\left[\frac{c_{3}}{\big(\EE(\omega_k^{1-a})\big)^{\frac{1}{1-a}} N^{\frac{1}{4(1-a)}}}, \frac{c_{1}}{N^{\frac{4}{\gamma(1-a)}}}\right],
\quad c_{3}:= \frac{2^{\frac{1}{1-a}}}{c_{2}^{\frac{1}{2(1-a)}}},
\end{equation*}
is non-empty $N\geqslant N_1$, where $N_1$, $c_1$, $c_2$ are from Theorem~\ref{th1osc}. For $N\geqslant N_1$ and $\e\in I_N$ the estimate
\begin{equation*}
\PP\left(\omega\in\Om:\, \lN{osc}-\L_0\leqslant N^{-\frac{1}{2}}\right)\leqslant N^{n\left(1-\frac{1}{\g}\right)}\E^{-c_{4}N^{\frac{n}{\g}}},
\end{equation*}
holds true, where constant $c_4>0$ depends only on distribution measure $\mu$.
\end{theorem}

\begin{theorem}\label{th3dlt}
Suppose that $\g\in\NN$, $\g\geqslant 17$ and the assumption of Theorem~\ref{th1dlt} is satisfied. Then the interval
\begin{equation*}
I_N:=\left[\frac{c_{3}}{\big(\EE(\omega_k^{\frac{1}{2}})\big)^2 N^{\frac{1}{2}}}, \frac{c_{1}}{N^{\frac{8}{\gamma}}}\right],
\quad c_{3}:= \frac{4}{c_{2}},
\end{equation*}
is non-empty $N\geqslant N_1$, where $N_1$, $c_1$, $c_2$ are from Theorem~\ref{th1dlt}. For $N\geqslant N_1$ and $\e\in I_N$ the estimate
\begin{equation*}
\PP\left(\omega\in\Om:\, \lN{dlt}-\L_0\leqslant N^{-\frac{1}{2}}\right)\leqslant N^{n\left(1-\frac{1}{\g}\right)}\E^{-c_{4}N^{\frac{n}{\g}}},
\end{equation*}
holds true, where constant $c_4>0$ depends only on distribution measure $\mu$.
\end{theorem}

The next three theorems are initial length scale estimates for operators $\OplN$, $\OpsN$, $\OpdN$.

\begin{theorem}\label{th4loc}
Suppose that $\alpha \in \Gamma$, $\g\in\NN$, $\g\geqslant 17$,
$N\in \NN$ and $\e\in I_N$, where $I_N$ is from Theorem~\ref{th3loc} and the assumption of Theorem~\ref{th1loc} is satisfied. We choose $\b_1,\b_2\in\G_{\a,N}$, $m_1, m_2>0$ so that $B_1:=\Pi_{\b_1,m_1}\subset\Pi_{\a,N}$, $B_2:=\Pi_{\b_2,m_2}\subset\Pi_{\a,N}$. Then there exists a constant $c_5>0$ independent of $\e$, $\a$, $N$, $\b_1$, $\b_2$, $m_1$, $m_2$ such that the inequality
\begin{align*}
\PP\bigg(\om\in\Om:\ & \|\chi_{B_1} (\OplN-\l)^{-1} \chi_{B_2} \|_{L_2(\Pi_{\a,N})\to L_2(\Pi_{\a,N})} \leqslant 2\sqrt{N} \E^{-\frac{c_{5}\dist(B_1,B_2)}{\sqrt{N}}}
\\
& \forall \, \lambda \leqslant \Lambda_0 + \frac{1}{2 \sqrt N}
\bigg)\geqslant 1-N^{\left(1-\frac{1}{\g}\right)} \E^{-c_{4} N^{\frac{1}{\g}}}
\end{align*}
holds true for
 $N\geqslant \max\{N_1^\gamma, N_2\}$, where $N_1$ is from Theorem~\ref{th1loc}, $N_2$ is from Theorem~\ref{th2loc}, $c_4$ is from Theorem~\ref{th3loc}.
\end{theorem}

\begin{theorem}\label{th4osc}
Suppose that $\alpha \in \Gamma$, $\g\in\NN$, $\g\geqslant 17$,
$N\in \NN$ and $\e\in I_N$, where $I_N$ is from Theorem~\ref{th3osc}, and the assumption of Theorem~\ref{th1osc} is satisfied. We choose $\b_1,\b_2\in\G_{\a,N}$, $m_1, m_2>0$ so that $B_1:=\Pi_{\b_1,m_1}\subset\Pi_{\a,N}$, $B_2:=\Pi_{\b_2,m_2}\subset\Pi_{\a,N}$. Then there exists a constant $c_5>0$ independent of $\e$, $\a$, $N$, $\b_1$, $\b_2$, $m_1$, $m_2$ such that the inequality
\begin{align*}
\PP\bigg(\om\in\Om:\ & \|\chi_{B_1} (\OpsN-\l)^{-1} \chi_{B_2} \|_{L_2(\Pi_{\a,N})\to L_2(\Pi_{\a,N})} \leqslant 2\sqrt{N} \E^{-\frac{c_{5}\dist(B_1,B_2)}{\sqrt{N}}}
\\
& \forall \, \lambda \leqslant \Lambda_0 + \frac{1}{2 \sqrt N}
\bigg)\geqslant 1-N^{n\left(1-\frac{1}{\g}\right)} \E^{-c_{4} N^{\frac{n}{\g}}}
\end{align*}
holds true for $N\geqslant \max\{N_1^\gamma, N_2\}$, where $N_1$ is from Theorem~\ref{th1osc}, $N_2$ is from Theorem~\ref{th2osc}, $c_4$ is from Theorem~\ref{th3osc}.
\end{theorem}

\begin{theorem}\label{th4dlt}
Suppose that $\alpha \in \Gamma$, $\g\in\NN$, $\g\geqslant 17$,
$N\in \NN$ and $\e\in I_N$, where $I_N$ is from Theorem~\ref{th3dlt}, and the assumption of Theorem~\ref{th1dlt} is satisfied. We choose $\b_1,\b_2\in\G_{\a,N}$, $m_1, m_2>0$ so that $B_1:=\Pi_{\b_1,m_1}\subset\Pi_{\a,N}$, $B_2:=\Pi_{\b_2,m_2}\subset\Pi_{\a,N}$. Then there exists a constant $c_5>0$ independent of $\e$, $\a$, $N$, $\b_1$, $\b_2$, $m_1$, $m_2$ such that the inequality
\begin{align*}
\PP\bigg(\om\in\Om:\ & \|\chi_{B_1} (\OpdN-\l)^{-1} \chi_{B_2} \|_{L_2(\Pi_{\a,N})\to L_2(\Pi_{\a,N})} \leqslant 2\sqrt{N} \E^{-\frac{c_{5}\dist(B_1,B_2)}{\sqrt{N}}}
\\
& \forall \, \lambda \leqslant \Lambda_0 + \frac{1}{2 \sqrt N}
\bigg)\geqslant 1-N^{n\left(1-\frac{1}{\g}\right)} \E^{-c_{4} N^{\frac{n}{\g}}}
\end{align*}
holds true for $N\geqslant \max\{N_1^\gamma, N_2\}$, where $N_1$ is from Theorem~\ref{th1dlt}, $N_2$ is from Theorem~\ref{th2dlt}, $c_4$ is from Theorem~\ref{th3dlt}.
\end{theorem}

Theorems~\ref{th1loc}--\ref{th4dlt} are adaption of the main results in
\cite{BGV15} to operators $\Op^{\e,\sharp}_{\a,N}$, $\sharp=\mathrm{loc, osc, dlt}$. They show how the general approach of work \cite{BGV15} can be extended for random perturbation not small w.r.t. the original operator, i.e., for non-regular perturbations. In the first two examples the presence of negative power of $\e$ in the definition of potentials $\Wloc$ and $\Wosc$ make the perturbation singular. In particular, potential $\Wosc$ is a classical example of perturbation in the homogenization theory \cite{BP}. The presence of a delta-interaction change the domain of the operator in comparison with the original one and it is singular in this sense. At the same time, as it is shown in the present work, these perturbation can be reduced to regular ones and then we can apply the approach of work \cite{BGV15}. The main idea is to use operators $\Vg{\sharp}$, $\sharp=\mathrm{loc, osc, dlt}$, see identities (\ref{4.7}), (\ref{5.4}), (\ref{6.1}). Keeping the spectrum, this operator transforms the original into a regular one, to which we can apply then the approach of work \cite{BGV15}.

We note that in the deterministic case the operator with large potentials localized on a set of a small measure were studied before, see, for instance, \cite{TMF05}, \cite{BiZhVM}. It was the motivation of considering random perturbation on the basis of such potentials.

It was shown in \cite[Ex. 7]{BGV15} that instead of layer $\Pi$, random operators (\ref{3.6}) with $V_0=0$ can be considered in a multi-dimensional case; the main result remain true. The same is true for our operators $\Op^{\e,\a,N}_{\sharp}(\om)$, ${\sharp=\mathrm{loc,osc,dlt}}$; for their analogues in multi-dimensional spaces Theorems~\ref{th1loc}--\ref{th4dlt} are also true.

\section{Preliminaries}

The proofs of Theorem~\ref{th1loc}--\ref{th4dlt} are based on the general approach developed in work \cite{BGV15}. This is why let us described the main results and the methods of this work.

We begin with the formulation of the problem. Let $\pL(t)$, $t\in[0,t_0]$, be a family of linear operators from $H^2(\square)$ into $L_2(\square)$ described by the formula
\begin{equation}\label{3.1}
\pL(t):=t \pL_1 + t^2 \pL_2 + t^3 \pL_3(t),
\end{equation}
where $\pL_i: H^2(\square)\to L_2(\square)$, $i=1,2,3$, are bounded symmetric operators, and operator $\pL_3(t)$ is assumed to be bounded uniformly in $t$. In \cite{BGV15} operators $\pL(t)$, $\pL_3(t)$ were defined for $t\in[-t_0,t_0]$. In our case it is sufficient to assume that defined just for $t\in[0,t_0]$. In order to satisfy formally the assumptions of work \cite{BGV15}, as
$-t_0\leqslant t<0$ we let $\pL(t):=\pL(-t)$, $\pL_3(t):=\pL_3(-t)$, so that $\pL(t)$, $\pL_3(t)$ happen to be defined for $t\in[-t_0,t_0]$.

Operators $\pL$, $\pL_1$, $\pL_2$, $\pL_3$ can be extended to operators acting from $H^2(\Pi)$ into $L_2(\Pi)$ as follows. For a function $u\in H^2(\Pi)$, its restriction on $\square$ belongs to $H^2(\square)$. This is why the action of operators $\pL$, $\pL_1$, $\pL_2$, $\pL_3$ is well-defined on this restriction and the result of the action is an element of $L_2(\square)$. We continue this element be zero in $\Pi\setminus\square$. The obtained function is the action of the required continuation of operators $\pL$, $\pL_1$, $\pL_2$, $\pL_3$ on the given function $u$. In what follows these operators are assumed to be continued in such a way. We observe that operators $\pL$, $\pL_1$, $\pL_2$, $\pL_3$ treated as operators in $L_2(\Pi)$ are generally speaking unbounded.

Let $\Op_\square$ be the operator $-\D+V_0$ in $\square$ subject to boundary condition (\ref{2.2}) on $\p\Pi\cap\p\square$ and to the Neumann condition on $\p\square\setminus\p\Pi$.

For operators $\pL_1$, $\pL_2$, in \cite{BGV15} there were made two main assumptions:

\begin{enumerate}\def\theenumi{A\arabic{enumi}}

\item\label{as1} The identity
\begin{equation*}
(\pL_1\psi_0,\psi_0)_{L_2(\square)}=0
\end{equation*}
holds true.

\item\label{as2} Let $U$ be the solution to the boundary value problem
\begin{equation}\label{3.2}
(\Op_\square-\L_0)U=\pL_1\psi_0
\end{equation}
orthogonal to $\psi_0$ in $L_2(\square)$. Assume that
\begin{equation}\label{3.3}
(\pL_2\psi_0,\psi_0)_{L_2(\square)} -(U,\pL_1\psi_0)_{L_2(\square)}>0.
\end{equation}
\end{enumerate}

By $\S(k)$ we denote the shift operator acting by the rule:
\begin{equation*}
(\S(k)u)(x)=u(x'-k,x_{n+1}).
\end{equation*}
We introduce the operator
\begin{equation}\label{3.6}
\Opg:=-\D+V_0+\sum\limits_{k\in\G_{\a,N}} \S(k) \pL(\e\om_k) \S(-k)
\end{equation}
in $L_2(\Pi_{\a,N})$ subject to boundary condition (\ref{2.2}) on the upper and lower boundaries and to boundary condition (\ref{2.9}). By $\lg$ we denote the minimal eigenvalue of operator $\Opg$.

Under assumptions~(\ref{as1}),~(\ref{as2}), in \cite{BGV15} there were proven the following four theorems.

\begin{theorem}\label{th1gen}
There exist positive constants $c_1$, $c_2$, $N_1$ such that for
\begin{equation*}
N\geqslant N_1\quad \text{и} \quad 0<\e<\frac{c_1}{N^4}
\end{equation*}
the estimate
\begin{equation*}
\lg-\L_0\geqslant \frac{c_2\e^2}{N^n}\sum\limits_{k\in \G_{\a,N}} \om_k^2
\end{equation*}
holds true.
\end{theorem}

\begin{theorem}\label{th2gen}
Suppose that $\a, \b_1, \b_2\in\G$, $m_1, m_2\in\NN$ are such that $B_1:=\Pi_{\b_1,m_1}\subset \Pi_{\a,N}$, ${B_2:=\Pi_{\b_2,m_2}\subset \Pi_{\a,N}}$. Then there exists $N_2\in\NN$ such that for $N\geqslant N_2$ the estimate
\begin{equation*}
\|\chi_{B_1}(\Opg -\l)^{-1}\chi_{B_2}\|_{L_2(\Pi_{\a,N})\to L_2(\Pi_{\a,N})} \leqslant \frac{C_1}{\d} \E^{-C_2\d \dist(B_1, B_2)}
\end{equation*}
holds true, where $\d:=\dist(\l,\spec(\Opg))>0$, $C_1$, $C_2$ are positive constants independent of $\e$, $\a$, $N$, $\d$, $\b_1$, $\b_2$, $m_1$, $m_2$, $\l$.
\end{theorem}

\begin{theorem}\label{th3gen}
Suppose that $\g\in\NN$, $\g\geqslant 17$. Then the interval
\begin{equation*}
I_N:=\left[\frac{c_{3}}{\EE(|\omega_k|) N^{\frac{1}{4}}}, \frac{c_{1}}{N^{\frac{4}{\gamma}}}\right],
\quad c_{3}:= \frac{2}{\sqrt c_{2}},
\end{equation*}
is non-empty $N\geqslant N_1$, where $N_1$, $c_1$, $c_2$ are from Theorem~\ref{th1gen}. For $N\geqslant N_1$ and $\e\in I_N$, the estimate
\begin{equation*}
\PP\left(\omega\in\Om:\, \lg-\L_0\leqslant N^{-\frac{1}{2}}\right)\leqslant N^{n\left(1-\frac{1}{\g}\right)}\E^{-c_{4}N^{\frac{n}{\g}}}
\end{equation*}
holds true, where constant $c_4>0$ depends only on the distribution measure $\mu$.
\end{theorem}

\begin{theorem}\label{th4gen}
Suppose that $\alpha \in \Gamma$, $\g\in\NN$, $\g\geqslant 17$,
 $N\in \NN$ and $\e \in I_N$. We choose $\b_1,\b_2\in\G_{\a,N}$, $m_1, m_2>0$ such that $B_1:=\Pi_{\b_1,m_1}\subset\Pi_{\a,N}$, $B_2:=\Pi_{\b_2,m_2}\subset\Pi_{\a,N}$. Then there exists a constant $c_5>0$ independent of $\e$, $\a$, $N$, $\b_1$, $\b_2$, $m_1$, $m_2$ such that the inequality
\begin{align*}
\PP\bigg(\om\in\Om:\ & \|\chi_{B_1} (\Opg-\l)^{-1} \chi_{B_2} \|_{L_2(\Pi_{\a,N})\to L_2(\Pi_{\a,N})} \leqslant 2\sqrt{N} \E^{-\frac{c_{5}\dist(B_1,B_2)}{\sqrt{N}}}
\\
& \forall \, \lambda \leqslant \Lambda_0 + \frac{1}{2 \sqrt N}
\bigg)\geqslant 1-N^{n\left(1-\frac{1}{\g}\right)} \E^{-c_{4} N^{\frac{n}{\g}}}
\end{align*}
holds true for $N\geqslant \max\{N_1^\gamma, N_2\}$, where $N_1$, $N_2$ is from Theorem~\ref{th1gen},~\ref{th2gen}, $c_4$ is from Theorem~\ref{th3gen}.
\end{theorem}

It was mentioned in \cite[Rem. 2.9]{BGV15} that operators $\pL_1$, $\pL_2$ can depend on $t$. We suppose that $\pL_1=\pL_1(t)$, $\pL_2=\pL_2(t)$, $t\in[0,t_0]$. For $t\in[-t_0,0)$ we redefine them as follows: $\pL_1(t)=\pL_1(-t)$, $\pL_2(t)=\pL_2(-t)$. These operators should be assumed to be uniformly bounded for $t\in[0,t_0]$ as operators from $H^2(\square)$ into $L_2(\square)$. Assumption~(\ref{as1}) should be satisfied for each $t\in[0,t_0]$, while estimate (\ref{3.3}) in assumption (\ref{as2}) should be replaced by the following one:
\begin{equation}\label{3.4}
(\pL_2(t)\psi_0,\psi_0)_{L_2(\square)} - (U,\pL_1(t)\psi_0)_{L_2(\square)} \geqslant c_0 >0,\quad t\in[0,t_0],
\end{equation}
where $c_0$ is a constant independent of $t$.

Let us stress certain features of the proofs of Theorems~\ref{th1gen},~\ref{th2gen},~\ref{th3gen},~\ref{th4gen}.

Theorem~\ref{th1gen} employs essentially the smallness of operator $\pL(\e\om_k)$ for small $\e$. At that, the symmetricity of this operator was not used in the proof; one just needed the reality of eigenvalue $\lg$. The only exclusion was the proof of an auxiliary estimate
\begin{equation}\label{3.5}
\L_0\leqslant \lg \leqslant \L_0+CN^{-2}
\end{equation}
for a given constant $C$. Under the presence of this estimate and the aforementioned reality of eigenvalue $\lg$, Theorem~\ref{th1gen} remains true for non-symmetric operators $\pL_1$, $\pL_2$ depending likely on $t$.

Theorem~\ref{th2gen} does not need the smallness of operator $\pL(\e \om_k)$ but employs the symmetricity. It also requires the self-adjointness of operator $\Opg$. The similar situation is for Theorems~\ref{th3gen},~\ref{th4gen}; they require just the symmetricity of operator $\pL(\e \om_k)$ and self-adjointness of и $\Opg$ as well as validity of Theorem~\ref{th1gen}. Operators $\pL_1$, $\pL_2$ can again depend on $t$.

Let us describe the scheme of the proof of Theorems~\ref{th1loc}--\ref{th4dlt}.
In view of the definition of operators $\Op^{\e,\a,N}_\sharp(\om)$, $\sharp=\mathrm{loc},\,\mathrm{osc},\,\mathrm{dlt}$, random perturbation in these operators can not be represented as (\ref{3.1}) that prevents a direct application of the results of work \cite{BGV15}. This is why for each of operators $\Op^{\e,\a,N}_\sharp(\om)$ we construct a special bounded and boundedly invertible operator $\Vg{\sharp}$ in $L_2(\Pi_{\a,N})$ such that the operator $\big(\Vg{\sharp}\big)^{-1}\Op^{\e,\a,N}_\sharp(\om)\Vg{\sharp}$ is represented as (\ref{3.6}). At that, we have to introduce a new small parameter and new random variables. Generally speaking, operators $\pL_i$ happen to be non-symmetric. But as it has been said above, this is a serious obstacle for proving Theorem~\ref{th1gen}; one just need to check the reality of eigenvalue $\l^{\e,\a,N}_\sharp(\om)$ and estimate (\ref{3.5}). It is clear that spectra of operators $\big(\Vg{\sharp}\big)^{-1}\Op^{\e,\a,N}_\sharp(\om)\Vg{\sharp}$ and $\Op^{\e,\a,N}_\sharp(\om)$ coincides. Thanks to the self-adjointness of the latter operator it ensures the reality of eigenvalue $\l^{\e,\a,N}_\sharp(\om)$. Then we succeed to prove estimate (\ref{3.5}) independently that finally leads us to the statement of Theorem~\ref{th1gen} for our particular operators $\Op^{\e,\a,N}_\sharp(\om)$. The formulation of the latter theorem for these operators is exactly Theorems~\ref{th1loc},~\ref{th1osc},~\ref{th1dlt}.

Then we return back to original operators
$\Op^{\e,\a,N}_\sharp(\om)$, where random perturbation are not small anymore but symmetric. And as it has been said above, this fact and proven Theorems~\ref{th1loc},~\ref{th1osc},~\ref{th1dlt} are sufficient to prove general theorems~\ref{th2gen},~\ref{th3gen},~\ref{th4gen}. Being applied to our operators, they give immediately Theorems~\ref{th2loc}--\ref{th4dlt}.

\section{Random localized potentia}

The present section is devoted to the study of operator $\OplN$ and the proof of Theorems~\ref{th1loc},~\ref{th2loc},~\ref{th3loc},~\ref{th4loc}.

We begin with proving Theorem~\ref{th1loc}. We observe first that by the self-adjointness of operator $\OplN$ its minimal eigenvalue is real. Then we transform operator $\OplN$ to (\ref{3.6}). We recall that we consider the case $n=1$.

Let $\Wl_*=\Wl_*(\xi)$ be the solution to the equation
\begin{equation}\label{4.1}
\frac{d^2 \Wl_*}{d\xi^2}=\Wl,\quad \xi\in\RR,
\end{equation}
determined by the formula
\begin{equation}\label{4.2}
\Wl_*(\xi)=\frac{1}{2}\int\limits_{\RR} |\xi-\z| \Wl(\z)\di z.
\end{equation}
We note that outside the support of $\Wl$, function $\Wl_*$ is linear:
\begin{equation}\label{4.3}
\Wl_*(\xi)=\frac{1}{2}\xi \int\limits_{\RR} \Wl(\z)\di\z - \frac{1}{2}\int\limits_{\RR} \z \Wl(\z)\di z
\end{equation}
to the right of the support of $\Wl$ and
\begin{equation}\label{4.4}
\Wl_*(\xi)=-\frac{1}{2}\xi \int\limits_{\RR} \Wl(\z)\di\z + \frac{1}{2}\int\limits_{\RR} \z \Wl(\z)\di z
\end{equation}
to the left of the support of $\Wl$. We let
\begin{equation}\label{4.5}
Q_{\mathrm{loc}}(x,\e,\om):=1+\sum\limits_{k\in\G_{\a,N}} (\e\om_k)^{2-a} \Wle(x_1-k,\e\om_k)\chi(x_1),
\end{equation}
where function $\Wle$ is introduced by the identities
\begin{equation*}
\Wle(x_1,\e):=\Wl_*\Big(\frac{x_1}{\e}\Big),\quad \e>0,\qquad \Wle(x_1,0):=0.
\end{equation*}
By $\chi=\chi(x_1)$ we denote an infinitely differentiable cut-off function equalling one in a neighborhood of the origin and vanishing outside a bigger neighborhood. The size of the bigger neighborhood is supposed to be small enough so that it is contained in
 $\square'$; we recall that by our assumption the origin is an internal point of $\square'$. In view of the identities and the presence of cut-off function $\chi$, the second term in the right hand side of (\ref{4.5}) is of order $O(\e^{1-a})$:
\begin{equation}\label{4.6}
\begin{aligned}
&\Big|(\e \om_k)^{2-a} \sum\limits_{k\in\G_{\a,N}} \Wle
(x_1-k,\e\om_k)\chi(x_1-k)\Big|\leqslant C\e^{1-a},
\\
&\Big|(\e \om_k)^{2-a} \frac{d\hphantom{x}}{dx_1} \sum\limits_{k\in\G_{\a,N}} \Wle
(x_1-k,\e\om_k)\chi(x_1-k)\Big|\leqslant C\e^{1-a},
\end{aligned}
\end{equation}
where constant $C$ is independent of $\e$, $x_1$, and $\om$. This is why the operator of multiplication by function $Q_{\mathrm{loc}}(x,\e,\om)$ is bounded and boundedly invertible in $L_2(\Pi)$. We denote such operator by $\Vg{loc}$. Since function $Q_{\mathrm{loc}}(\cdot,\e,\om)$ belongs to $C^2(\overline{\Pi})$, is independent of $x_{n+1}$ and is identically equals to one in the vicinity of the lateral boundary of $\Pi_{\a,N}$, operator $\Vg{loc}$ maps the domain of operator $\OplN$ onto itself. Employing equation (\ref{4.1}), by straightforward calculations one can check easily that
\begin{equation}\label{4.7}
\begin{aligned}
\big(&\Vg{loc}\big)^{-1}\OplN\Vg{loc}=-\D+V_0
\\
&+\sum\limits_{k\in\G_{\a,N}} (\e \om_k)^{1-a} \left(A_{1,\mathrm{loc}}(x_1-k,\e\om_k) \frac{d\hphantom{x}}{dx_1} + A_{0,\mathrm{loc}}(x_1-k,\e\om_k) \right),
\end{aligned}
\end{equation}
where the operator in the right hand side is considered in $\Pi_{\a,N}$ with the same boundary conditions as $\OplN$. Coefficients $A_{0,\mathrm{loc}}^{\e\om_k}(x_1-k)$, $A_{1,\mathrm{loc}}^{\e\om_k}(x_1-k)$ are determined by the identities
\begin{align*}
A_{1,\mathrm{loc}}(x_1,\e):=&-\frac{\e}{1+\e^{2-a}\Wle(x_1,\e)\chi(x_1)} \frac{d\hphantom{x}}{dx_1} \Wle(x_1,\e)\chi(x_1),
\\
A_{0,\mathrm{loc}}(x_1,\e):=&-\frac{\e}{1+\e^{2-a}\Wle(x_1,\e)\chi(x_1)} \bigg(2\frac{d\Wle(x_1,\e)}{dx_1} \frac{d\chi}{dx_1}(x_1)
\\
&+\Wle(x_1,\e) \frac{d^2\chi}{dx_1^2}(x_1)\bigg)
\\
&+\frac{\e^{1-a}}{1+\e^{2-a}\Wle(x_1,\e)\chi(x_1)}\chi(x_1)\Wle(x_1,\e)\Wloc(x_1,\e).
\end{align*}
These formulae and estimates (\ref{4.6}) imply that coefficients $A_{0,\mathrm{loc}}(x_1,\e\om_k)$, $A_{1,\mathrm{loc}}(x_1,\e\om_k)$ are bounded uniformly in $x_1$, $\e$, $\om$. This is why the operator in the right hand side of identity (\ref{4.7}) can be represented as (\ref{3.6}) if we take $\e^{\frac{1-a}{2}}$ as a new small parameter, $\om_k^{\frac{1-a}{2}}$ as new random variables, and (\ref{3.1}) we let
\begin{equation}\label{4.9}
\pL_1:=0,\quad \pL_2:=3,\quad \pL_2:=K_{1,\mathrm{loc}}(x_1,t)\frac{d\hphantom{x}}{dx_1}+K_{0,\mathrm{loc}}(x_1,t),
\end{equation}
where coefficients $K_{1,\mathrm{loc}}$, $K_{0,\mathrm{loc}}$ are determined by the formulae:
\begin{equation}\label{4.10}
\begin{aligned}
K_{1,\mathrm{loc}}(x_1,t):=&-\frac{t^{\frac{1}{1-a}}}{1+t^{\frac{2-a}{1-a}} \Wl_*\left(\frac{x_1}{t^{\frac{1}{1-a}}}\right)}
\frac{d\hphantom{x}}{dx_1} \Wl_*\left(\frac{x_1}{t^{\frac{1}{1-a}}}\right)\chi(x_1),
\\
K_{0,\mathrm{loc}}(x_1,t):=&-\frac{t^{\frac{1}{1-a}}}{1+t^{\frac{2-a}{1-a}} \Wl_*\left(\frac{x_1}{t^{\frac{1}{1-a}}}\right)} \bigg(2 \frac{d\chi}{dx_1}(x_1)\frac{d\hphantom{x}}{dx_1} \Wl_*\left(\frac{x_1}{t^{\frac{1}{1-a}}}\right) \\
&+\frac{d^2\chi}{dx_1^2}(x_1)\Wl_*\left(\frac{x_1}{t^{\frac{1}{1-a}}}\right) \bigg)
\\
&+\frac{1}{1+t^{\frac{2-a}{1-a}}\Wl_*\left(\frac{x_1}{t^{\frac{1}{1-a}}}\right)} \chi(x_1)\Wl_*\left(\frac{x_1}{t^{\frac{1}{1-a}}}\right) \Wl\left(\frac{x_1}{t^{\frac{1}{1-a}}}\right),
\end{aligned}
\end{equation}
as $t>0$ and
\begin{equation}\label{4.8}
K_{1,\mathrm{loc}}(x_1,0):=0,\quad K_{0,\mathrm{loc}}(x_1,0):=\frac{1}{2}\int\limits_{\RR} \Wl(\z)\di\z.
\end{equation}
The choice of the values for coefficients $K_{1,\mathrm{loc}}(x_1,0)$, $K_{0,\mathrm{loc}}(x_1,0)$ is arbitrary since $\pL(0)=0$. The above choice of these values will be clarified later, cf, Remark~\ref{rm4.1}.

Let us prove that operator $\pL(t)$ introduced by formulae (\ref{3.1}), (\ref{4.9}) satisfies Assumptions~(\ref{as1}), (\ref{as2}). The first of them is satisfied since $\pL_1=0$. To check the other, we first observe that for our case the solution to equation (\ref{3.2}) is zero: $U=0$. This is why to check inequality (\ref{3.4}), it is sufficient to estimate from below the scalar $(\pL_2(t)\psi_0,\psi_0)_{L_2(\square)}$. Since coefficients $A_1$, $A_0$ are real-valued, the same is true for this scalar product. Formulae (\ref{4.10}), estimates (\ref{4.6}), identities  (\ref{4.4}), (\ref{4.5}), and the fact that the supports of the functions $\Wl\left(\frac{x_1}{t^{\frac{1}{1-a}}}\right)$ and $1-\chi(x_1)$ are disjoint for small $t$ imply immediately that
\begin{equation}\label{4.11}
\begin{aligned}
&K_{1,\mathrm{loc}}(x_1,t)\frac{d\hphantom{x_1}}{dx_1}\psi_0(x_{n+1})=0,
\\
&(K_{0,\mathrm{loc}}\psi_0,\psi_0)_{L_2(\square)} =\int\limits_{\square'}K_{0,\mathrm{loc}}(x_1,t)\di x_1= \int\limits_{\RR}A_0(x_1,t)\di x_1
\\
&\hphantom{(K_{0,\mathrm{loc}}\psi_0,\psi_0)_{L_2(\square)}} = t^{\frac{1}{1-a}}\int\limits_{\RR} \frac{d\hphantom{x}}{dx_1}\big(1-\chi(x_1)\big) \Wl_*\left(\frac{x_1}{t^{\frac{1}{1-a}}}\right)\di x_1+O(t^{\frac{1}{1-a}})
\\
&\hphantom{(K_{0,\mathrm{loc}}\psi_0,\psi_0)_{L_2(\square)}} =\int\limits_{\RR} \Wl(\z)\di\z+O(t^{\frac{1}{1-a}}).
\end{aligned}
\end{equation}
These relations, Assumption~(\ref{2.19}) and definition (\ref{4.9}) of operator $\pL_2$ yield required estimate (\ref{3.4}) with $c_0=\frac{1}{2}\int\limits_{\RR} \Wl(\z)\di\z$.

\begin{remark}\label{rm4.1}
The above choice of value for $A_0(x_1,0)$ ensures estimate (\ref{3.4}) for $t=0$ with above mentioned constant $c_0$.
\end{remark}

In view of said in the previous section, to complete the proof of Theorem~\ref{th1loc} we just need to check estimates (\ref{3.5}). By the minimax principle for the original self-adjoint operator $\OplN$ with test function $\psi_0$ we have
\begin{align*}
\lN{loc}\leqslant & \frac{\|\nabla\psi_0\|_{L_2(\Pi_{\a,N})}^2 +(V_0\psi_0,\psi_0)_{L_2(\Pi_{\a,N})}}{\|\psi_0\|_{L_2(\Pi_{\a,N})}^2}
\\
&+\frac{\sum\limits_{k\in\G_{\a,N}} \big(\Wloc(\cdot-k, \e\om_k)\psi_0,\psi_0\big)_{L_2(\Pi_{\a,N})}} {\|\psi_0\|_{L_2(\Pi_{\a,N})}^2}
\\
\leqslant & \L_0 + \frac{\sum\limits_{k\in\G_{\a,N}} \big(\Wloc(\cdot-k, \e\om_k)\psi_0,\psi_0\big)_{L_2(\Pi_{\a,N})}} {\sum\limits_{k\in\G_{\a,N}}\|\psi_0\|_{L_2(\square)}^2}
\\
\leqslant & \L_0 + \frac{\sum\limits_{\genfrac{}{}{0 pt}{}{k\in\G_{\a,N}}{\e\om_k\not=0}} (\e\om_k)^{-a} \left(\Wl\left(\frac{\cdot}{\e\om_k}\right)\psi_0,\psi_0\right)_{L_2(\square)}} {\sum\limits_{k\in\G_{\a,N}}\|\psi_0\|_{L_2(\square)}^2}
\\
\leqslant & \L_0+\frac{\e^{1-a}}{|\square'|} \int\limits_{\RR} \Wl(\z)\di\z\leqslant \L_0 + \frac{C}{N^8},
\end{align*}
and for sufficiently great $N_1$ (cf. (\ref{2.10})) we arrive at the right estimate in (\ref{3.5}).

To prove the left estimate in (\ref{3.5}), in domain $\Pi_{\a,N}$ we consider lateral boundaries $\p\square_k\setminus\p\Pi$ of sets $\square_k$ for each $k\in\G_{\a,N}$, and on these surfaces we impose Neumann boundary conditions. Then by the minimax principle, eigenvalue $\lN{loc}$ is estimated from below by the minimal among smallest eigenvalues of operators $\Op^{\e,\mathrm{loc}}_{k,1}(\om_k)$, $k\in\G_{\a,N}$, on cells $\square_k$:
\begin{equation*}
\lN{loc}\geqslant \min\limits_{k\in\G_{\a,N}} \l^{\e,\mathrm{loc}}_{k,1}(\om_k).
\end{equation*}
Smallest eigenvalue $\l^{\e,\mathrm{loc}}_{k,1}(\om_k)$ of operator $\Op^{\e,\mathrm{loc}}_{k,1}(\om_k)$ is also the smallest eigenvalue of operator $\big(\mathcal{V}^{\e,\mathrm{loc}}_{k,1}\big)^{-1}
\Op^{\e,\mathrm{loc}}_{k,1}(\om_k)\mathcal{V}^{\e,\mathrm{loc}}_{k,1}$. According to (\ref{4.7}) with $\a=k$, $N=1$, this operator is a small regular perturbation of operator $-\D+V_0$ in $\square_k$ subject to boundary condition (\ref{2.2}) on $\p\square_k\cap\p\Pi$ and to Neumann condition on $\p\square_k\setminus\p\Pi$. This is why in accordance with the general theory of regular perturbations, $\l^{\e,\mathrm{loc}}_{k,1}(\om_k)$ has the asymptotics
\begin{equation}\label{4.12}
\begin{aligned}
\l^{\e,\mathrm{loc}}_{k,1}(\om_k)=&\L_0
+ \frac{(\e \om_k)^{1-a}}{|\square'|} \left(
\left(A_{1,\mathrm{loc}}(\,\cdot,\e\om_k)\frac{d\hphantom{x}}{dx_1}+ A_{0,\mathrm{loc}}(\,\cdot,\e\om_k)\right)\psi_0,\psi_0\right)_{L_2(\square)} \\
&+ O((\e\om_k)^{2-2a}).
\end{aligned}
\end{equation}
Formulae (\ref{4.10}) with $t=(\e\om_k)^{\frac{1-a}{2}}$ yield that
\begin{equation*}
\left(
\left(A_{1,\mathrm{loc}}(\,\cdot, \e\om_k)\frac{d\hphantom{x}}{dx_1}+ A_{0,\mathrm{loc}}(\,\cdot,\e\om_k)\right)\psi_0,\psi_0\right)_{L_2(\square)} =\int\limits_{\RR} \Wl(\z)\di\z+ O(\e\om_k),
\end{equation*}
and hence, asymptotics (\ref{4.12}) becomes
\begin{equation*}
\l^{\e,\mathrm{loc}}_{k,1}(\om_k)=\L_0
+ \frac{(\e \om_k)^{1-a}}{|\square'|}\int\limits_{\RR} \Wl(\z)\di\z+ O((\e\om_k)^{2-2a}).
\end{equation*}
Now by Assumption~(\ref{2.19}) we arrive at the left estimate in (\ref{3.5}). The proof of Theorem~\ref{th1loc} is complete.

The proofs of Theorems~\ref{th2gen},~\ref{th3gen},~\ref{th4gen} for operator $\OplN$ are borrowed from \cite{BGV15} with no changes and it leads us to Theorems~\ref{th2loc},~\ref{th3loc},~\ref{th4loc}.

\begin{remark}\label{rm4.2}
We observe that we consider operator $\OplN$ with random localized potential only in a strip assuming $n=1$. In the multi-dimensional case we can also construct transformation $\Vg{loc}$ satisfying formula (\ref{4.7}). Such transformation should be constructed as (\ref{4.5}) and function $\Wl_*$ should be introduced as the solution to the equation
\begin{equation*}
\D_\xi \Wl_*=\Wl,\quad \xi\in\RR^n,
\end{equation*}
determined by the identity
\begin{equation*}
\Wl_*(\xi):=-\int\limits_{\RR^n} E(\xi-\z) \Wl(z)\di z,
\end{equation*}
where $E$ is the fundamental solution of Laplace operator $\RR^n$. At the same time, after passing to the transformed operator, Assumption~(\ref{as2}) is not satisfied, namely, estimate (\ref{3.4}) fails. This is the reason for introducing the aforementioned restriction for the dimension of layer $\Pi$.
\end{remark}

\section{Random fast oscillating potential}

In the present section we consider operator $\OpsN$ and prove Theorems~\ref{th1osc},~\ref{th2osc},~\ref{th3osc},~\ref{th4osc}. The scheme of the proof follows the same lines as in the third section: we pay the main assumption to the proof of Theorem~\ref{th1gen} for operator~\ref{th1osc}. After that, the proof of Theorems~\ref{th2gen},~\ref{th3gen},~\ref{th4gen} is borrowed from \cite{BGV15} with no changes and being applied to operator $\OpsN$, it gives the statements of Theorems~\ref{th2osc},~\ref{th3osc},~\ref{th4osc}. This is why in what follows we prove Theorem~\ref{th1osc} only.

Thanks to the self-adjointness of operator $\OpsN$, its smallest eigenvalue $\lN{osc}$ is real. Let us construct operator $\Vg{osc}$ transforming operator $\OpsN$ to (\ref{3.6}). We let
\begin{equation}\label{5.1}
Q_{\mathrm{osc}}(x,\e,\om):=1+\sum\limits_{k\in\G_{\a,N}} (\e\om_k)^{2-a} \Wse(x'-k,x_{n+1},\e\om_k),
\end{equation}
where function $\Wse$ is determined by the identities:
\begin{equation}\label{5.2}
\Wse(x,\e):=\Ws_*\left(x,\frac{x}{\e}\right),\quad\e>0,\qquad \Ws(x,0):=0.
\end{equation}

By $\Vg{osc}$ we denote the operator of multiplication by function $Q_{\mathrm{osc}}(x,\e,\om)$. Due to the smoothness of $\Ws$, function $Q_{\mathrm{osc}}$ is twice continuously differentiable w.r.t. $x$ in $\overline{\Pi}$. Moreover, uniform in $x\in\overline{\Pi}$, $\e$, $\om$ estimates
\begin{equation}\label{5.3}
\begin{aligned}
&\Big|(\e \om_k)^{2-a} \sum\limits_{k\in\G_{\a,N}} \Wse
(x'-k,x_{n+1},\e\om_k)\chi(x'-k)\Big|\leqslant C\e^{2-a},
\\
&\Big|(\e \om_k)^{2-a}\nabla_{x'} \sum\limits_{k\in\G_{\a,N}} \Wse
(x'-k,x_{n+1},\e\om_k)\chi(x'-k)\Big|\leqslant C\e^{1-a}
\end{aligned}
\end{equation}
hold true. This is why operator $\Vg{osc}$ is bounded and boundedly invertible in $L_2(\Pi_{\a,N})$ and maps the domain of operator $\OpsN$ onto itself. As in (\ref{4.7}), in view of equation (\ref{2.12}), one can easily check that
\begin{equation}\label{5.4}
\begin{aligned}
\big(\Vg{osc}\big)^{-1}&\OplN\Vg{osc}=-\D+V_0+
\\
&+\sum\limits_{k\in\G_{\a,N}} (\e \om_k)^{1-a} \bigg(\sum\limits_{j=1}^{n+1}A_{j,\mathrm{osc}}(x'-k,x_{n+1},\e\om_k) \frac{\p\hphantom{x}}{\p x_j}+
\\
&\hphantom
{+
\sum\limits_{k\in\G_{\a,N}} (\e \om_k)^{1-a} \bigg(\sum\limits_{j=1}^{n+1}} + A_{0,\mathrm{osc}}(x'-k,x_{n+1},\e\om_k) \bigg).
\end{aligned}
\end{equation}
Operator in the right hand side of this identity is considered in $\Pi_{\a,N}$ with the same boundary conditions as $\OpsN$. Coefficients $A_{j,\mathrm{osc}}$, $A_{0,\mathrm{osc}}$ read as
\begin{align*}
A_{j,\mathrm{osc}}(x,\e):=&-\frac{\e}{1+\e^{2-a}\Wse(x,\e)}\frac{\p\Wse}{\p x_j}(x,\e),
\\
A_{0,\mathrm{osc}}(x,\e):=&-\frac{1}{1+\e^{2-a}\Wse(x,\e)} \bigg(2\sum\limits_{j=1}^{n+1} \frac{\p^2\Ws_*}{\p x_j\p\xi_j}\left(x,\frac{x}{\e}\right)+ \e (\D_x \Ws_*)\left(x,\frac{x}{\e}\right)
\\
&+\e^{1-a} \Ws_*\left(x,\frac{x}{\e}\right) \Wse(x,\e)\bigg) + \e^{1-a} W(x).
\end{align*}
The first two terms in the brackets in the right hand side of the formula for $A_{0,\mathrm{osc}}$ should be treated in the sense of the partial derivatives w.r.t. $x$ and $\xi$ for function $\Ws_*(x,\xi)$ followed by the substitution $\xi=\frac{x}{\e}$.

It follows from estimates (\ref{5.3}) that functions $A_{j,\mathrm{osc}}(x,\e)$, $A_{0,\mathrm{osc}}(x,\e)$ are bounded uniformly in $x$, $\e$ and $\om_k$. The right hand side of identity (\ref{5.4}) can be represented as (\ref{3.6}) satisfying at the same time Assumptions~(\ref{as1}), (\ref{as2}).

In order to do it, we shall make use of the following auxiliary lemma.

\begin{lemma}\label{lm5.1}
Suppose that function $w=w(x,\xi)$ defined in $\RR^{2n+2}$ is 1-periodic w.r.t. each of variables $\xi_i$, $i=1,\ldots,n$, and is compactly supported w.r.t. $x$:
\begin{equation}\label{5.6}
\supp w(\cdot,\xi)\subseteq M\subset \square \quad\text{for each}\quad \xi\in\RR^n,
\end{equation}
where $M$ is a some fixed set. Suppose that
\begin{equation}\label{5.7}
\frac{\p^{|\a|+|\b|}w}{\p x^\a \p\xi^\b} \in C(\RR^{2n+2}),\quad \a,\b\in\ZZ_+^n,\quad |\a|\leqslant m, \quad |\b|\leqslant 1,
\end{equation}
for some $m\in\NN$. Then the asymptotic identity
\begin{equation}\label{5.8a}
\int\limits_{\square} w\left(x,\frac{x}{\e}\right)\di x=\int\limits_{\square}\di x\int\limits_{(0,1)^{n+1}} w(x,\xi)\di\xi + O(\e^m)
\end{equation}
holds true.
\end{lemma}

\begin{proof}
Passing to the function
\begin{equation*}
(x,\xi)\mapsto w(x,\xi)\ -\int\limits_{(0,1)^{n+1}} w(x,\z)\di \z,
\end{equation*}
we see that it is sufficient to prove the statement of the lemma for the case
\begin{equation}\label{5.15}
\int\limits_{(0,1)^{n+1}} w(x,\xi)\di\xi=0\quad\text{для}\quad x\in\square.
\end{equation}
Thanks to this identity, the boundary value problem for the equation
\begin{equation*}
\D_{\xi} w_*=w,\quad \xi\in(0,1)^{n+1}
\end{equation*}
subject to periodic boundary conditions is solvable for each
$x\in\square$ and there exists the unique solution satisfying condition (\ref{5.15}). This function possesses the following smoothness:
\begin{equation}\label{5.16}
\frac{\p^{|\a|+|\b|}w_*}{\p x^\a \p\xi^\b} \in C(\RR^{2n+2}),\quad \a,\b\in\ZZ_+^n,\quad |\a|\leqslant m, \quad |\b|\leqslant 2.
\end{equation}
As $w$, function $w_*$ is compactly supported w.r.t. $x$. By the equation for $w_*$, the identity
\begin{equation*}
\e^2\sum\limits_{j=1}^{n+1}\frac{\p}{\p x_j} \frac{\p w_*}{\p\xi_j} \left(x,\frac{x}{\e}\right)=\e\sum\limits_{j=1}^{n+1} \frac{\p^2 w_*}{\p x_j \p\xi_j} \left(x,\frac{x}{\e}\right) + w\left(x,\frac{x}{\e}\right)
\end{equation*}
holds true, where the derivatives in the right hand side are treated as the partial derivatives w.r.t. $x$ and $\xi$ for function $w(x,\xi)$, and the derivative w.r.t. in $x_j$ in the left hand side is the total derivative w.r.t. $x_j$ for a function depending of $x$ and $x/\e$. In view of the last identity we have:
\begin{equation}\label{5.17}
\begin{aligned}
\int\limits_{\square} w\left(x,\frac{x}{\e}\right) \di x=&\e^2\int\limits_{\square} \sum\limits_{j=1}^{n+1}\frac{\p}{\p x_j} \frac{\p w_*}{\p\xi_j} \left(x,\frac{x}{\e}\right) \di x-
 \e\int\limits_{\square} \sum\limits_{j=1}^{n+1} \frac{\p^2 w_*}{\p x_j \p\xi_j} \left(x,\frac{x}{\e}\right)\di x
\\
=& - \e\int\limits_{\square} \sum\limits_{j=1}^{n+1} \frac{\p^2 w_*}{\p x_j \p\xi_j} \left(x,\frac{x}{\e}\right)\di x.
\end{aligned}
\end{equation}
We observe that each of the integrands in the left hand side of the obtained identity has smoothness (\ref{5.7}) with $m$ replaced by $m-1$ and satisfies condition (\ref{5.15}). Applying identity (\ref{5.17}) as many times as needed, we arrive at the statement of the lemma.
\end{proof}

We denote
\begin{equation*}
\Tosc(\e):=\frac{\e^{a-1}}{|\square'|}\int\limits_{\square} \frac{\psi_0^2(x_{n+1})}{1+\e^{2-a}\Wse(x,\e)} \bigg(2\sum\limits_{j=1}^{n+1} \frac{\p^2\Ws_*}{\p x_j\p\xi_j}\left(x,\frac{x}{\e}\right)+ \e (\D_x \Ws_*)\left(x,\frac{x}{\e}\right)\bigg)\di x.
\end{equation*}
By the first estimate in (\ref{5.3}), the identity
\begin{equation}\label{5.11}
\frac{1}{1+\e^{2-a}\Wse(x,\e)}=1+O(\e^{2-a})
\end{equation}
holds true uniformly in $x\in\overline{\square}$. This is why Lemma~\ref{lm5.1}, condition (\ref{2.20}) and the smoothness of function $\Ws_*$ imply the identity
\begin{equation}\label{5.8}
\Tosc(\e)=O(\e).
\end{equation}

Let us define operator $\pL(t)$. We let
\begin{equation}\label{5.9}
\begin{aligned}
&\pL_1(t):=\sum\limits_{j=1}^{n+1} K_{j,\mathrm{osc}}^{(1)}(x,t)\frac{\p\hphantom{x}}{\p x_j} + K_{0,\mathrm{osc}}^{(1)}(x,t),
\\
&\pL_2(t):=K_{0,\mathrm{osc}}^{(2)}(x,t),\quad \pL_3(t):=0,
\end{aligned}
\end{equation}
as $t>0$, where
\begin{align*}
K_{j,\mathrm{osc}}^{(1)}(x,t):=&\frac{t^{\frac{1}{1-a}}}{1 +t^{\frac{2-a}{1-a}}\Wse(x,t^{\frac{1}{1-a}})} \frac{\p\hphantom{x}}{\p x_j}\Wse(x,t^{\frac{1}{1-a}}),
\\
K_{0,\mathrm{osc}}^{(1)}(x,t):=&-\frac{1}{1 +t^{\frac{2-a}{1-a}}\Wse(x,t^{\frac{1}{1-a}})} \bigg(2\sum\limits_{j=1}^{n+1} \frac{\p^2\Ws_*}{\p x_j\p\xi_j}\left(x,\frac{x}{t^{\frac{1}{1-a}}}\right)
\\
&+ t^{\frac{1}{1-a}} (\D_x \Ws_*)\left(x,\frac{x}{t^{\frac{1}{1-a}}}\right)\bigg) + t\,\Tosc(t^{\frac{1}{1-a}}),
\\
K_{0,\mathrm{osc}}^{(2)}(x,t):=&-\Tosc(t^{\frac{1}{1-a}}) + W(x) + \frac{\Ws\left(x,\frac{x}{t^{\frac{1}{1-a}}}\right) \Wse\left(x,\frac{x}{t^{\frac{1}{1-a}}}\right)}{1 +t^{\frac{2-a}{1-a}}\Wse(x,t^{\frac{1}{1-a}})},
\end{align*}
and for $t=0$, operators $\pL_i$ are determined by the formulae
\begin{equation}\label{5.10}
\begin{aligned}
\pL_1(0):=&0,\quad \pL_3(0):=0,
\\
\pL_2(0):=&\frac{1}{2} \int\limits_{\square} W(x)\psi_0^2(x_{n+1})\di x
\\
&-\frac{1}{2}\int\limits_{\square} \di x\, \psi_0^2(x_{n+1}) \int\limits_{(0,1)^{n+1}} |\nabla_{\xi} \Ws_*(x,\xi)|^2\di\xi.
\end{aligned}
\end{equation}
It is easy to make sure that under such choice of operator $\pL(t)$, the right hand side of (\ref{5.4}) becomes (\ref{3.6}), if as a new small parameter we choose $\e^{1-a}$, and as new random variables we take $\om_k^{1-a}$.

Let us check Assumptions~(\ref{as1}), (\ref{as2}). The first of the assumptions follows directly from the definition of quantity $\Tosc$ and coefficient $K_{0,\mathrm{osc}}^{(2)}$.

Let us check Assumption~(\ref{as2}), namely, estimate (\ref{3.4}). First we find out the behavior of scalar product $(\pL_2(t)\psi_0,\psi_0)_{L_2(\square)}$. In order to do it, we employ estimate (\ref{5.8}), identity (\ref{5.11}) and Lemma~\ref{lm5.1}:
\begin{align*}
(\pL_2&(t)\psi_0,\psi_0)_{L_2(\square)}=\int\limits_{\square} K_{2,\mathrm{osc}}^{(2)}(x,t)\psi_0^2(x_{n+1})\di x
=\int\limits_{\square} W(x)\psi_0^2(x_{n+1})\di x
\\
&+ \int\limits_{L_2(\square)} \psi_0^2(x_{n+1}) \Ws\left(x,\frac{x}{t^{\frac{1}{1-a}}}\right) \Wse(x,t^{\frac{1}{1-a}})\di x+O(t^{\frac{1}{1-a}})
\\
=& \int\limits_{\square} W(x)\psi_0^2(x_{n+1})\di x
\\
&+ \int\limits_{L_2(\square)} \psi_0^2(x_{n+1}) \Ws\left(x,\frac{x}{t^{\frac{1}{1-a}}}\right) \Ws_*\left(x,\frac{x}{t^{\frac{1}{1-a}}}\right) \di x+O(t^{\frac{1}{1-a}})
\\
=& \int\limits_{\square} W(x)\psi_0^2(x_{n+1})\di x
+\int\limits_{\square} \di x\, \psi_0^2(x_{n+1}) \int\limits_{(0,1)^{n+1}} \Ws(x,\xi) \Ws_*(x,\xi)\di\xi +O(t^{\frac{1}{1-a}}).
\end{align*}
In view of equation (\ref{2.12}) and boundary conditions for $\Ws_*$, we can integrate by parts in the latter integral:
\begin{align*}
\int\limits_{\square} & \di x\, \psi_0^2(x_{n+1}) \int\limits_{(0,1)^{n+1}} \Ws(x,\xi) \Ws_*(x,\xi)\di\xi
\\
 &= \int\limits_{\square} \di x\, \psi_0^2(x_{n+1}) \int\limits_{(0,1)^{n+1}} \Ws_*(x,\xi)\D_{\xi} \Ws_*(x,\xi)\di\xi
\\
 &= -\int\limits_{\square} \di x\, \psi_0^2(x_{n+1}) \int\limits_{(0,1)^{n+1}} |\nabla_\xi \Ws_*(x,\xi)|^2 \Ws_*(x,\xi)\di\xi.
\end{align*}
We finally obtain:
\begin{equation}\label{5.12}
\begin{aligned}
&(\pL_2(t)\psi_0,\psi_0)_{L_2(\square)} =\int\limits_{L_2(\square)} W(x)\psi_0^2(x_{n+1})\di x
\\
& -\int\limits_{\square} \di x\, \psi_0^2(x_{n+1}) \int\limits_{(0,1)^{n+1}} |\nabla_\xi \Ws_*(x,\xi)|^2 \Ws_*(x,\xi)\di\xi + O(t^{\frac{1}{1-a}}).
\end{aligned}
\end{equation}

Let us find out the behavior of the solution to equation (\ref{3.2}). The right hand side of this equation is the function
\begin{align*}
\big(\pL_1(t)\psi_0\big)(x,t)=&-\frac{\psi_0(x_{n+1})}{1 +t^{\frac{2-a}{1-a}}\Wse(x,t^{\frac{1}{1-a}})} \bigg(2\sum\limits_{j=1}^{n+1} \frac{\p^2\Ws_*}{\p x_j\p\xi_j}\left(x,\frac{x}{t^{\frac{1}{1-a}}}\right)
\\
&+ t^{\frac{1}{1-a}} (\D_x \Ws_*)\left(x,\frac{x}{t^{\frac{1}{1-a}}}\right)\bigg) + t\,\Tosc(t^{\frac{1}{1-a}})\psi_0(x_{n+1}).
\end{align*}
For the solution to equation (\ref{3.2}) with such right hand side, one can construct its asymptotic expansion for small $t$ by the multiscale method \cite{BP}. This expansion is valid at least in the norm of $L_2(\square)$. The leading term of this expansion is a quantity of order $O(t^{\frac{2}{1-a}})$. This is why
\begin{equation*}
(U,\pL_1(t)\psi_0)_{L_2(\square)}=O(t^{\frac{2}{1-a}}),\quad t\to0.
\end{equation*}
Hence, in view of (\ref{5.12}) we have:
\begin{equation}\label{5.13}
\begin{aligned}
(\pL_2(t)&\psi_0,\psi_0)_{L_2(\square)}-(U,\pL_1(t)\psi_0)_{L_2(\square)}=
 \int\limits_{L_2(\square)} W(x)\psi_0^2(x_{n+1})\di x
\\
& -\int\limits_{\square} \di x\, \psi_0^2(x_{n+1}) \int\limits_{(0,1)^{n+1}} |\nabla_\xi \Ws_*(x,\xi)|^2 \Ws_*(x,\xi)\di\xi + O(t^{\frac{1}{1-a}})
\\
\geqslant & \frac{1}{2} \int\limits_{L_2(\square)} W(x)\psi_0^2(x_{n+1})\di x
\\
&- \frac{1}{2} \int\limits_{\square} \di x\, \psi_0^2(x_{n+1}) \int\limits_{(0,1)^{n+1}} |\nabla_\xi \Ws_*(x,\xi)|^2 \Ws_*(x,\xi)\di\xi
\end{aligned}
\end{equation}
for sufficiently small $t$. It proves Assumption (\ref{as2}) for $t>0$. As $t=0$, we have $\pL_1(t)\psi_0=0$, $U=0$, and estimate (\ref{5.13}) is satisfied by the definition of $\pL_2(0)$, cf. (\ref{5.10}).

Let us check estimates (\ref{3.5}). As in the previous section, to prove the right estimate, we apply the minimax principle with test function $\psi_0$:
\begin{align*}
\lN{osc}\leqslant & \frac{\|\nabla\psi_0\|_{L_2(\Pi_{\a,N})}^2 +(V_0\psi_0,\psi_0)_{L_2(\Pi_{\a,N})}}{\|\psi_0\|_{L_2(\Pi_{\a,N})}^2}
\\
&+\frac{\sum\limits_{k\in\G_{\a,N}} \big(\Wosc(\cdot-k,x_{n+1},\e\om_k)\psi_0,\psi_0\big)_{L_2(\Pi_{\a,N})}} {\|\psi_0\|_{L_2(\Pi_{\a,N})}^2}
\\
\leqslant & \L_0 + \frac{\sum\limits_{k\in\G_{\a,N}} \big(\Wosc(\cdot-k,x_{n+1},\e\om_k)\psi_0,\psi_0\big)_{L_2(\Pi_{\a,N})}} {\sum\limits_{k\in\G_{\a,N}}\|\psi_0\|_{L_2(\square)}^2}
\\
\leqslant & \L_0 + \frac{\sum\limits_{\genfrac{}{}{0 pt}{}{k\in\G_{\a,N}}{\e\om_k\not=0}} (\e\om_k)^{-a} \left(\Wosc\left(\cdot,\frac{\cdot}{\e\om_k}\right)\psi_0,\psi_0\right)_{L_2(\square)}} {\sum\limits_{k\in\G_{\a,N}}\|\psi_0\|_{L_2(\square)}^2}
\\
= & \L_0 + \frac{1}{N^n|\square'|} \sum\limits_{\genfrac{}{}{0 pt}{}{k\in\G_{\a,N}}{\e\om_k\not=0}} (\e\om_k)^{-a} \int\limits_{\square} \Ws\left(x,\frac{x}{\e\om_k}\right)\psi_0^2(x_{n+1})\di x.
\end{align*}
By Lemma~\ref{lm5.1} and Assumption~(\ref{2.18}) it yields:
\begin{align*}
\lN{osc}\leqslant \L_0 + C\e^{1-a}\leqslant \L_0+ \frac{C}{N^4}, \quad C>0,
\end{align*}
where constant $C$ is independent of $\e$ and $N$. In view of (\ref{2.13}), for sufficiently large $N_1$ it implies the right estimate in (\ref{3.5}).

In order to prove the left estimate in (\ref{3.5}), completely by analogy with the previous section on the basis of the minimax principle we get the lower estimate:
\begin{equation}\label{5.14}
\lN{osc}\geqslant \min\limits_{k\in\G_{\a,N}} \l^{\e,\mathrm{osc}}_{k,1}(\om_k),
\end{equation}
where $\l^{\e,\mathrm{osc}}_{k,1}(\om_k)$ is the smallest eigenvalue of operator $\Op^{\e,\mathrm{osc}}_{k,1}(\om_k)$ on cell $\square_k$ subject to appropriate boundary conditions. At the same time, it is an eigenvalue of the operator $\big(\mathcal{V}^{\e,\mathrm{osc}}_{k,1}(\om_k)\big)^{-1} \Op^{\e,\mathrm{osc}}_{k,1}(\om_k)\mathcal{V}^{\e,\mathrm{osc}}_{k,1}(\om_k)$. In accordance with identity (\ref{5.4}), the latter is a small regular perturbation of the operator $-\D+V_0$ in $\square_k$ subject to appropriate boundary conditions. Then in accordance with regular perturbation theory, the asymptotics for eigenvalue $\l^{\e,\mathrm{osc}}_{k,1}(\om_k)$ reads as
\begin{align*}
\l^{\e,\mathrm{osc}}_{k,1}(\om_k)=&\L_0 + \frac{(\e\om_k)^{1-a}}{|\square'|} (\pL_1(\e\om_k)\psi_0,\psi_0)_{L_2(\square)}
\\
&+\frac{(\e\om_k)^{2-2a}}{|\square'|} \left(
(\pL_2(\e\om_k)\psi_0,\psi_0)_{L_2(\square)} -(U,\pL_1(\e\om_k)\psi_0)_{L_2(\square)}
\right)
\\
& + O\big((\e\om_k)^{3-3a}\big),
\end{align*}
where $U$ is the solution to equation (\ref{as2}) with right hand side $\pL_1(\e\om_k)\psi_0$. By formula (\ref{5.13}) and Assumption (\ref{as1}) it follows that
\begin{align*}
\l^{\e,\mathrm{osc}}_{k,1}(\om_k)=&\L_0 + \frac{(\e\om_k)^{2-2a}}{|\square'|} \bigg(\int\limits_{\square} W(x)\psi_0^2(x_{n+1})\di x
\\
&-\int\limits_{\square} \di x\, \psi_0^2(x_{n+1}) \int\limits_{(0,1)^{n+1}} |\nabla_{\xi} \Ws_*(x,\xi)|^2\di\xi
\bigg)+O\big((\e\om_k)^{3-3a}\big)
\end{align*}
that implies
\begin{equation*}
\l^{\e,\mathrm{osc}}_{k,1}(\om_k)\geqslant \L_0
\end{equation*}
for sufficiently small $\e$. By (\ref{5.14}) it proves the leftt estimate in (\ref{3.5}) and completes the proof of Theorem~\ref{th1osc}.

\begin{remark}\label{rm5.1}
The idea of constructing operator $\Vg{osc}$ is borrowed from works \cite{TMF06}, \cite{TMF07}, where a similar operator was constructing in one-dimensional case. In the present work this approach is extended to an arbitrary dimension.
\end{remark}

\section{Random delta-interaction}

In the present section we study operator $\OpdN$ and prove Theorems~\ref{th1dlt},~\ref{th2dlt},~\ref{th3dlt},~\ref{th4dlt}.

It was shown in papers \cite[Ex. 5]{MPAG07}, \cite[Ex. 5]{AHP07} that by means of a certain change of spatial variables and a multiplication by a certain function, the operators with delta-interactions can be reduced to usual differential operators keeping the spectrum and the self-adjointness. Applying the results of works \cite[Ex. 5]{MPAG07}, \cite[Ex. 5]{AHP07} to our operator $\OpdN$, we arrive at the following statement.

\begin{lemma}\label{lm6.1}
There exists a change of variables $y=(y_1,\ldots,y_{n+1})$, $y_i=y_i(x,\e\om)$, $i=1,\ldots,n+1$, such that
\begin{enumerate}
\item\label{lm6.1it1} Outside small fixed neighborhoods of surfaces $S_k$, $k\in \G$, change $x\mapsto y$ is identical, i.e., $y_i=x_i$. Change $x\mapsto y$ maps each cell $\square_k$ onto itself.

\item\label{lm6.1it2} Functions $y_i$ are twice differentiable functions and their second derivatives are piecewise continuous.

\item\label{lm6.1it3} Let $p=p(y,\e\om)$ be the Jacobian of change $x\mapsto y$, i.e., \begin{equation*}
 p=\det \frac{D(y_1,\ldots,y_{n+1})}{D(x_1,\ldots,x_{n+1})},
 \end{equation*}
 where the matrix in the left hand side is the Jacobi matrix of the change. On functions $u\in L_2(\Pi_{\a,N})$, $u=u(y)$ we define the operator in terms of the inverse change $x=x(y,\e\om)$:
 \begin{equation*}
 (\Vg{dlt}u)(x):=p^{-\frac{1}{2}}(x(y,\e\om)) u(x(y,\e\om)).
 \end{equation*}
 The identity
 \begin{equation}
 \begin{aligned}
 \big(\Vg{dlt}\big)^{-1}\OpdN\Vg{dlt}=&-\D+V_0
 \\
 &+\sum\limits_{k\in\G_{\a,N}} \e\om_k\S(k)\mathcal{M}(\e\om_k)\S(-k)
 \end{aligned}\label{6.1}
 \end{equation}
 holds true, where $M(t)$ is a symmetric second order differential operator with piecewise coefficients vanishing outside a small neighborhood of surface $S$.

 \item\label{lm6.1it4} The identity
 \begin{equation*}
 (\mathcal{M}(t)\psi_0,\psi_0)_{L_2(\square)}=\int\limits_{S} \Wd \psi_0^2\di S
 \end{equation*}
 holds true.
\end{enumerate}
\end{lemma}

We define operator $\pL(t)$:
\begin{equation*}
\pL_1:=0,\quad \pL_2(t):=\mathcal{M}(t^{\frac{1}{2}}),\quad \pL_3(t):=0,\quad \pL(t):=t\mathcal{M}(t), \quad t\in[0,t_0].
\end{equation*}
In accordance with Statement~\ref{lm6.1it3} of Lemma~\ref{lm6.1}, the left hand side of identity (\ref{6.1}) coincides with operator (\ref{3.6}) if as a new small parameter we take $\e^{\frac{1}{2}}$, and $\om_k^{\frac{1}{2}}$ as new random variables. At that, Assumption~(\ref{as1}) holds true. The solution of the corresponding equation (\ref{3.2}) vanishes, and thus, by (\ref{2.15}), Assumption~(\ref{as2}) is also satisfied. And since operators $\pL_i$, $i=1,2,3$, are symmetric, we can apply directly the general results of Theorems~\ref{th1gen},~\ref{th2gen},~\ref{th3gen},~\ref{th4gen}. It leads us immediately to Theorems~\ref{th1dlt},~\ref{th2dlt},~\ref{th3dlt},~\ref{th4dlt}.

\section*{Acknowlegments}

The authors thank G.P.~Panasenko for discussing particular aspects of the work. The research is supported by the grant of Russian Scientific Foundation (project no. 14-11-00078).

\end{document}